\documentclass[12pt]{amsart}

\usepackage{enumitem}
\usepackage{stmaryrd}
\usepackage{url}
\usepackage[dvipsnames]{xcolor}
\usepackage{tikz-cd}
\usepackage{graphicx}
\usepackage{amsmath}
\usepackage{color}
\usepackage{caption}
\usepackage{subcaption}
\usepackage{amsfonts,amssymb,amscd}
\usepackage{mathrsfs}
\usepackage{bm}
\usepackage{verbatim} % needed for comment enviroment
\usepackage{manfnt}
\usepackage{yhmath}

\usepackage[toc,page]{appendix}
\calclayout

\newtheoremstyle{remboldstyle}
  {}{}{\itshape}{}{\bfseries}{.}{.5em}{{\thmname{#1 }}{\thmnumber{#2}}{\thmnote{ (#3)}}}
\theoremstyle{remboldstyle}

\usepackage[outdir=./]{epstopdf}

\newtheorem{thm}{Theorem}[section]
\newtheorem{prop}[thm]{Proposition}
\newtheorem{lem}[thm]{Lemma}
\newtheorem{cor}[thm]{Corollary}

\newtheorem{thmx}{Theorem}
\theoremstyle{definition}
\newtheorem{definition}[thm]{Definition}

\newtheorem{rem}[thm]{Remark}
\newtheorem{notation}[thm]{Notation}

\usepackage[greek,english]{babel}
\usepackage[LGR,T1]{fontenc}
\usepackage[utf8]{inputenc}

\makeatletter
\DeclareFontFamily{U}{tipa}{}
\DeclareFontShape{U}{tipa}{m}{n}{<->tipa10}{}
\newcommand{\arc@char}{{\usefont{U}{tipa}{m}{n}\symbol{62}}}%

\newcommand{\arc}[1]{\mathpalette\arc@arc{#1}}

\newcommand{\arc@arc}[2]{%
  \sbox0{$\m@th#1#2$}%
  \vbox{
    \hbox{\resizebox{\wd0}{\height}{\arc@char}}
    \nointerlineskip
    \box0
  }%
}
\makeatother

\numberwithin{equation}{section}

\catcode`\@=11
\newdimen\cdsep
\def\cdstrut{\vrule height .6\cdsep width 0pt depth .4\cdsep}
\def\@cdstrut{{\advance\cdsep by 2em\cdstrut}}

\def\arrow#1#2{
  \ifx d#1
    \llap{$\scriptstyle#2$}\left\downarrow\cdstrut\right.\@cdstrut\fi
  \ifx u#1
    \llap{$\scriptstyle#2$}\left\uparrow\cdstrut\right.\@cdstrut\fi
  \ifx r#1
    \mathop{\hbox to \cdsep{\rightarrowfill}}\limits^{#2}\fi
  \ifx l#1
    \mathop{\hbox to \cdsep{\leftarrowfill}}\limits^{#2}\fi
}
\catcode`\@=12

\newcommand{\Chat}{\widehat{\mathbb{C}}}

\makeatletter
\@namedef{subjclassname@2020}{\textup{2020} Mathematics Subject Classification}
\makeatother

\begin{document}

%%
%% The title of the paper goes here.  Edit to your title.
%%

\title[Blaschke Products Sharing a Welding Homeomorphism]{Blaschke Products Sharing a Welding Homeomorphism}

%%
%% Now edit the following to give your name and address:
%% 
\author{Kirill Lazebnik}
\thanks{\noindent The author was supported by NSF grant DMS-2452130.} 

% The changes the MR classification year from 1991 ro 2020
% This is not necessary if the latest version of TeX is installed
%\makeatletter
%\@namedef{subjclassname@2020}{\textup{2020} Mathematics Subject Classification}
%\makeatother

%%
%% If there is another author uncomment and edit the following.
%%

%\author{Second Author}
%\address{Department of Mathematics, University of South Carolina,
%Columbia, SC 29208}
%\email{second@math.sc.edu}
%\urladdr{www.math.sc.edu/$\sim$second}

%%
%% If there are three of more authors they are added in the obvious
%% way. 
%%

%%%
%%% The following is for the abstract.  The abstract is optional and
%%% if not used just delete, or comment out, the following.
%%%

\begin{abstract} Given two Blaschke products $A$, $B$ of the same degree, a \emph{welding homeomorphism} of $A$, $B$ is a circle homeomorphism $k: \mathbb{T} \rightarrow \mathbb{T}$ satisfying $A\circ k=B$ on $\mathbb{T}$. We show that any two pairs of Blaschke products sharing a welding homeomorphism must factor through a common pair of Blaschke products, and we give a criterion for detecting whether two Blaschke products admit such a factorization. %The proof is a short application of the uniformization theorem and L\"uroth's theorem.
\end{abstract}

%%
%%  LaTeX will not make the title for the paper unless told to do so.
%%  This is done by uncommenting the following.
%%
%\date{\today}

\maketitle

%%
%% LaTeX can automatically make a table of contents.  This is done by
%% uncommenting the following:
%%
%\tableofcontents

\section{Introduction}

The equation in the following Definition \ref{weldingeqtncurves} has recently been featured in several studies on applications of Complex Analysis to Pattern Recognition. We denote by $\Chat:=\mathbb{C}\cup\{\infty\}$ the Riemann sphere, $\mathbb{D}:=\{z\in\Chat: |z|<1\}$, $\mathbb{D}^*:=\{z\in\Chat: |z|>1\}$, $\mathbb{T}:=\{z \in\Chat : |z|=1\}$, and we recall that conformal maps between Jordan domains extend to homeomorphisms between their boundaries, by Carath\'eodory's Theorem. 
%Let us describe the question in the recent literature on applications of Complex Analysis to Pattern Recognition which we claim Theorem \ref{soly_ans2} answers. 

\begin{definition}\label{weldingeqtncurves} The \emph{Welding Equation for Jordan Curves} is the functional equation 
\begin{equation}\label{exieqtndefncurves} \phi^{\gamma}\circ k = \psi^{\gamma}
\end{equation}
where $k: \mathbb{T}\rightarrow\mathbb{T}$ is a given homeomorphism, and a solution is an oriented Jordan curve $\gamma\subset\Chat$ so that (\ref{exieqtndefncurves}) holds on $\mathbb{T}$ where $\psi^{\gamma}$ (resp. $\phi^{\gamma}$) is a conformal map of $\mathbb{D}$ (resp. $\mathbb{D}^*$) onto the component of $\Chat\setminus\gamma$ oriented positively (resp. negatively) with respect to $\gamma$.%, and $k:\mathbb{T}\rightarrow\mathbb{T}$ is a homeomorphism.
\end{definition}

\begin{rem} Circle homeomorphisms $k$ for which there exists a solution to $\phi^{\gamma}\circ k = \psi^{\gamma}$ are called \emph{conformal weldings} in the Complex Analysis literature, and smooth conformal weldings are called \emph{fingerprints} in the recent literature on Pattern Recognition.

% \emph{fingerprints} in the recent literature on Pattern Recognition, or \emph{conformal weldings} in the Complex Analysis literature. %The homeomorphism $k$ is generally given and a solution to the welding equation is a Jordan curve $\gamma$ so that $\phi^{\gamma}\circ k = \psi^{\gamma}$ holds on $\mathbb{T}$. Circle homeomorphisms $k$ for which there exists a solution to $\phi^{\gamma}\circ k = \psi^{\gamma}$ are called \emph{fingerprints} in the recent literature on Pattern Recognition, or \emph{welding homeomorphisms} in the literature on conformal welding. 
\end{rem} 

%Given a Jordan curve $\gamma$ in the plane or sphere, one may associate an essentially unique circle homeomorphism to $\gamma$ called the \emph{fingerprint} of $\gamma$

%\begin{definition}\label{fingerprint_defn} Let $\gamma\subset\Chat$ be a smooth Jordan curve, and denote by $\phi_{\textrm{in}}$ (resp. $\phi_{\textrm{out}}$) the conformal map of $\mathbb{D}:=\{z : |z|<1\}$ (resp. $\mathbb{D}^*:=\{z: |z|>1\}$) onto one component of $\Chat\setminus\gamma$ (resp. the other component). The homeomorphism $k_\gamma:=\phi_{\textrm{out}}^{-1}\circ\phi_{\textrm{in}}: \mathbb{T}\rightarrow\mathbb{T}$ is termed the \emph{fingerprint} of $\gamma$.
%\end{definition}

Fingerprints lend a particularly useful metric for studying the space of smooth curves (via a choice of metric on the space of smooth homeomorphisms $\mathbb{T}\mapsto\mathbb{T}$), a point of view initiated in \cite{MR902292}, \cite{Sharon-Mumford} and continued in \cite{MR2858965}, \cite{MR3784168}, \cite{solynin2020fingerprintslemniscatesquadraticdifferentials} among others. %To explain our point of view, let us introduce a bit of notation.

The following equation was introduced in \cite{MR2868587}; its relevance to the study of fingerprints is explained in Theorems \ref{younsi_result}, \ref{khav_result} below.

\begin{definition}\label{welding_defn_bp} The \emph{Welding Equation for Blaschke Products} is the functional equation 
\begin{equation}\label{exieqtndefn} A\circ k = B
\end{equation}
where $k:\mathbb{T}\rightarrow\mathbb{T}$ is a given homeomorphism and a solution is a pair of Blaschke products $A$, $B$ of the same degree so that (\ref{exieqtndefn}) holds on $\mathbb{T}$.
\end{definition}

%\begin{rem} As for the Welding Equation for Jordan curves, the circle homeomorphism $k$ is given and a solution is a pair of Blaschke products $A$, $B$ satisfying $A\circ k = B$ on $\mathbb{T}$.
%\end{rem}

We recall a Blaschke product is a rational map $r(z)=p(z)/q(z)$ of degree $\geq1$ which, when restricted to $\mathbb{D}$, is a proper map onto $\mathbb{D}$. For rational $r$, we denote the \emph{lemniscate} of $r$ by $L_r:=r^{-1}(\mathbb{T})$. 

%\noindent The terms rational mapping and rational function are used interchangeably.

%\begin{rem} One may equivalently define $\textrm{Rat}_d$ as the space of holomorphic self-mappings of the Riemann Sphere $\Chat$ having the property that each point $w\in\Chat$ has $d$-many preimages, counting multiplicity.
%\end{rem}

%\begin{thm}\label{khav_result} \cite{MR2868587} Let $p$ be a polynomial of degree $d$, assume $L_p$ is a Jordan curve, and denote by $k_p$ the fingerprint of $L_p$. Then there exists a solution $A$, $B$ to $A\circ k_p =B$ with $A(z)=z^d$, and moreover this solution is essentially unique in that any other solution $A$, $\tilde{B}$ with $A(z)=z^d$ satisfies $\tilde{B}=e^{i2k\pi/d}B$ for some $0\leq k<d$. 
%\end{thm}

\begin{thm}\label{younsi_result} \cite{MR3447662} There exists a solution to $A\circ k=B$ if and only if $k$ is a fingerprint of a Jordan curve rational lemniscate. 
\end{thm}

The choice to study fingerprints of lemniscates is explained by Hilbert's Lemniscate Theorem: Hilbert's theorem asserts that Jordan curve lemniscates are dense (in a suitable sense) among all planar curves; we refer to \cite{MR4880202} for further discussion. This connection was observed in \cite{MR2868587} where Theorem \ref{younsi_result} was conjectured, and the following polynomial version of Theorem \ref{younsi_result} was proven with $A$ fixed to be $A(z)=z^d$. 

\begin{thm}\label{khav_result} \cite{MR2868587} There exists a solution to $A\circ k =B$ with $A(z)=z^d$ if and only if $k$ is a fingerprint of a Jordan curve polynomial lemniscate, in which case a solution $z^d$, $B$ is unique up to multiplication of $B$ by $e^{i2k\pi/d}$ for $0\leq k<d$. 
\end{thm}

% Let $k$ be the fingerprint of a Jordan curve polynomial lemniscate $L_p$. Then there exists a solution $A$, $B$ to $A\circ k =B$ with $A(z)=z^d$ where $d:=\deg(p)$, and any other solution $A$, $\tilde{B}$ with $A(z)=z^{d}$ satisfies $\tilde{B}=e^{i2k\pi/d}B$ for some $0\leq k<d$.
%\end{thm}

%\begin{thm}\label{khav_result} For each degree $d$ Jordan curve polynomial lemniscate, there exists a unique Blaschke product $B$ satisfying (\ref{exieqtndefn}) with $k$ the fingerprint of the lemniscate and $A(z):=z^d$. 
%\end{thm} 

%It was proven in \cite{MR2868587} that the fingerprints of (Jordan curve) polynomial lemniscates $p^{-1}(\mathbb{T})$ coincide with %solutions to the functional equation
%\begin{equation}\label{khavfunceqtn} (z\mapsto z^n) \circ k = B \end{equation}

%Moreover, it is proven in \cite{MR2868587} that $k\mapsto B$ is also onto, although ... solutions $k$ to (\ref{exieqtndefn}) with $A(z):=z^d$ and $B$ a Blaschke product coincide with fingerprints of Jordan curve polynomial lemniscates.

% For variable $A$, one has the following. %Thus the coefficients of $B$ provide coordinates on (a dense subspace) of the collection of planar curves. 

We supply a uniqueness conclusion to Theorem \ref{younsi_result} which, as we will see, differs from the corresponding uniqueness conclusion of Theorem \ref{khav_result}. Indeed, it is evident that if $A$, $B$ solves $A\circ k=B$ then so does $C\circ A$, $C\circ B$ for any Blaschke product $C$. We prove that this is the only non-uniqueness of solutions to $A\circ k=B$.

\begin{thmx}\label{uniquenessconcmalik} Two pairs $A$, $B$ and $A^*$, $B^*$ of Blaschke products share the same welding homeomorphism if and only if there exist Blaschke products $A_1$, $B_1$ and $C$, $C^*$ so that
\begin{equation}\label{matrix_version2} \begin{bmatrix} A & B \\ A^* & B^* \end{bmatrix} = \begin{bmatrix}  C \circ A_1 \hspace{5mm}  C \circ B_1 \\
 C^* \circ A_1 \hspace{4mm} C^* \circ B_1 \end{bmatrix}
\end{equation} % $A$, $A^*$ both factor through $A_1$ and $B$, $B^*$ both factor through $B_1$. 
\end{thmx}

\noindent We prove Theorem \ref{uniquenessconcmalik} in Section \ref{proof_section}: the main tools are the Uniformization Theorem and L\"uroth's Theorem. The following Corollary \ref{myyounsi_result} may be deduced from Theorem \ref{uniquenessconcmalik}.

\begin{cor}\label{myyounsi_result} If a solution $A$, $B$ to the welding equation $A\circ k = B$ with homeomorphism $k$ exists, then the solution $A$, $B$ of minimal degree $d$ is essentially unique in that every other degree $d$ solution must be of the form $M\circ A$, $M\circ B$ for some $M\in\emph{M\"ob}_{\mathbb{D}}$.
\end{cor}

\noindent Combining Theorem \ref{uniquenessconcmalik} with Theorem \ref{younsi_result}, one has the following.  

\begin{thm}\label{def_answer} A solution to the welding equation with homeomorphism $k$ exists if and only if $k$ is a fingerprint of a Jordan curve rational lemniscate $L$, in which case any solution must be of the form $C\circ A$, $C\circ B$ where $A$, $B$ is a minimal degree solution. 
\end{thm}

Now: suppose one has a solution $A$, $B$ to $A\circ k=B$. If $A$, $B$ are of minimal degree then according to Theorem \ref{def_answer} the full solution space for $A\circ k =B$ is simply the collection $C\circ A$, $C\circ B$ over all Blaschke products $C$. This leads to the question: how can one tell whether a solution $A$, $B$ to $A\circ k =B$ is of minimal degree? Equivalently: how can one detect whether a given $A$, $B$ admit factorizations $A=C\circ A_1$, $B=C\circ B_1$ where $C$ is a Blaschke product of degree $>1$? We give a necessary and sufficient criterion in Theorem \ref{mon_thm} below based on the monodromy groups of $A$, $B$, which we recall the definition of now (see \cite{MR3753897}, Chapter 9 for a more detailed discussion). We denote by $S_d$ the group of permutations of $\{1, ..., d\}$. 

%\begin{notation}\label{epsilon_not} Let $\varepsilon>0$ be sufficiently small so that $\{z : 1-\varepsilon<|z|<1\}$ contains no critical values neither of $A$ nor of $B$. Set $\zeta:=1-\varepsilon/2$, and let $\{a_1, ..., a_d\}:=A^{-1}(\zeta)$ and $\{b_1, ..., b_d\}:=B^{-1}(\zeta)$, with an ordering chosen so that the points $a_i$ (resp. $b_i$) appear consecutively in a counterclockwise-parametrization around the Jordan curve $A^{-1}(|\zeta|\mathbb{T})$ (resp. $B^{-1}(|\zeta|\mathbb{T})$). Denote by $S_d$ the group of permutations of $\{1, ..., d\}$. 
%\end{notation}

\begin{rem}\label{mon_group_defn} Assume $\zeta\in\mathbb{D}$ is not a critical value of $A$. Let $\{a_1, ..., a_d\}:=A^{-1}(\zeta)$. The \emph{monodromy group} of $A$ (based at $\zeta)$ is defined to be the following subgroup of $S_d$. Let $\gamma: [0,1]\rightarrow\mathbb{D}$ be a loop in $\mathbb{D}$ with $\gamma(0)=\gamma(1)=\zeta$ which avoids the critical values of $A$: in particular this means that $A^{-1}$ may be analytically continued along $\gamma$. Covering properties of $A$ imply that $A^{-1}(\gamma)\subset \mathbb{D}$ consists of $d$ many oriented arcs intersecting only at their endpoints: for each $i$ there is exactly one arc of $A^{-1}(\gamma)$ beginning at $a_i$ and exactly one arc of $A^{-1}(\gamma)$ terminating at $a_i$ (these two arcs may be the same in which case $A^{-1}(\gamma)$ contains a loop based at $a_i$). The loop $\gamma$ defines a permutation $\gamma^*\in S_d$ as follows; for each $i\in\{1,..., d\}$ we set $\gamma^*(i):=j$ where $a_j$ is the terminal point of the unique arc in $A^{-1}(\gamma)$ which begins at $a_i$. The collection of $\gamma^*$ for all such loops $\gamma$ based at $\zeta$ and avoiding the critical values of $A$ is defined to be the \emph{monodromy group} of $A$. The monodromy group of $B$ is defined similarly (assuming also $\zeta$ is not a critical value of $B$ and fixing a labelling $\{b_1, ..., b_d\}:=B^{-1}(\zeta)$). We denote by $\gamma_A^*$ (resp. $\gamma_B^*$) the monodromy group element of $A$ (resp. $B$) induced by $\gamma$.
\end{rem}
%\begin{rem}\label{mongroupdefn} The \emph{Monodromy group} of $B$ is defined to be the following subgroup of $S_d$. Let $\gamma: [0,1]\rightarrow\mathbb{D}$ be a loop in $\mathbb{D}$ with $\gamma(0)=\gamma(1)=\zeta$ which avoids the critical values of $B$: in particular this means that $B^{-1}$ may be analytically continued along $\gamma$. Covering properties of $B$ imply that $B^{-1}(\gamma)\subset \mathbb{D}$ consists of $d$ many oriented arcs intersecting only at their endpoints: for each $i$ there is exactly one arc of $B^{-1}(\gamma)$ beginning at $b_i$ and exactly one arc of $B^{-1}(\gamma)$ terminating at $b_i$ (these two arcs may be the same in which case $B^{-1}(\gamma)$ contains a loop based at $b_i$). The loop $\gamma$ defines a permutation $\gamma^*\in S_d$ as follows; for each $i\in\{1,..., d\}$ we set $\gamma^*(i):=j$ where $b_j$ is the terminal point of the unique arc in $B^{-1}(\gamma)$ which begins at $b_i$. The collection of $\gamma^*$ for all such loops $\gamma$ based at $\zeta$ and avoiding the critical values of $B$ is defined to be the \emph{Monodromy group} of $B$. 
%\end{rem}

\noindent Recall a \emph{relation} on $\{1, ..., d\}$ is a subset of $\{1, ..., d\}\times\{1, ..., d\}$. 

\begin{definition}\label{pre_reltn_defn} We define a relation $\simeq$ on  $\{1, ..., d\}$ as follows: if there exists $\gamma$ with $\gamma_B^*(i)=j$, $\gamma_A^*(i)=k$, we set $j\simeq k$ and $k\simeq j$. We define $\sim$ to be the equivalence relation on $\{1, ..., d\}$ generated by $\simeq$: that is, we set $x\sim y$ iff there is a chain 
\begin{equation}\label{chain_defn} x_0\simeq x_1\simeq x_2\simeq...\simeq x_n \end{equation}
where $x_0=x$ and $x_n=y$. We say $\sim$ is \emph{trivial} if $j \sim k$ for all $j$, $k\in\{1, ..., d\}$, otherwise we say $\sim$ is \emph{non-trivial}.
\end{definition}

\noindent Lastly, we remark that $\sim$ depends on the ordering in $a_1, ..., a_d\in A^{-1}(\zeta)$ and $b_1, ..., b_d \in B^{-1}(\zeta)$, and whether $\sim$ is trivial or non-trivial may change if $a_1, ..., a_d$ and $b_1, ..., b_d$ are re-ordered. 

\begin{thmx}\label{mon_thm} Two Blaschke products $A$, $B$ satisfy $A=C\circ A_1$, $B=C\circ B_1$ for some Blaschke products $A_1$, $B_1$, $C$ with $\deg(C)>1$ if and only if there exists an ordering of $A^{-1}(\zeta)$, $B^{-1}(\zeta)$ for which $\sim$ is non-trivial.
\end{thmx}

We prove Theorem \ref{mon_thm} in Section \ref{mainsection}: the main tool is a result of \cite{cowen2012finiteblaschkeproductscompositions} asserting that a Blaschke product $B$ admits a decomposition $B=B_2\circ B_1$ if and only if there exists an equivalence relation $\sim$ which is respected by the monodromy group of $B$ (in a sense defined in Section \ref{mainsection}). The proof of Theorem \ref{mon_thm} then amounts to showing that $\sim$ of Definition \ref{pre_reltn_defn} is respected by both $A$, $B$ and that the second factors in the resulting decompositions for $A$, $B$ coincide.

%\begin{rem} The following Lemmas \ref{rittsthmpart1'}, \ref{rittsthm} are due to \cite{cowen2012finiteblaschkeproductscompositions}. The work \cite{cowen2012finiteblaschkeproductscompositions} is unpublished for reasons explained therein; an excellent account of \cite{cowen2012finiteblaschkeproductscompositions} may be found in \cite{MR3753897}, and related ideas are also found in the much older \cite{MR1501189}. We include proofs of Lemmas \ref{rittsthmpart1'}, \ref{rittsthm} (following \cite{MR3753897} which in turn follows \cite{cowen2012finiteblaschkeproductscompositions}) nevertheless for the reader's sake, as the ideas are central to the overall proof of Theorem \ref{uniquenessconcmalik}.
%\end{rem}

%\begin{rem} For generic choices of $A$, $B$ the equivalence relation $\sim$ is trivial: for instance if $\textrm{CV}(A)\cap\textrm{CV}(B)=\emptyset$ then $\sim$ is trivial. 
%\end{rem}

\section{Proof of Theorem \ref{uniquenessconcmalik}.}\label{proof_section}

The $\impliedby$ direction in the proof of Theorem \ref{uniquenessconcmalik} is clear: if $A$, $B$ and $A^*$, $B^*$ satisfy (\ref{matrix_version2}), then $A$, $B$ and $A^*$, $B^*$ both share the same welding homeomorphism which coincides with the welding homeomorphism of $A_1$, $B_1$. 

We turn our attention to the proof of the $\implies$ direction. The following Theorem \ref{maliksthm} is essentially Theorem 3.1 of \cite{MR3447662} (and, in turn, proves Theorem \ref{younsi_result}); the last assertion of Theorem \ref{maliksthm} is not present in the statement of \cite{MR3447662} but follows from the proof.

\begin{notation}\label{phi_notation} Suppose $r$ is a rational function so that $r^{-1}(\mathbb{T})$ is a Jordan curve, $1\in r^{-1}(\mathbb{T})$, and $0$, $\infty$ are in different components of $\Chat\setminus r^{-1}(\mathbb{T})$. We will denote by $\phi_{\textrm{in}}$ (resp. $\phi_{\textrm{out}}$) the Riemann map of $\mathbb{D}$ (resp. $\mathbb{D}^*$) onto $\{z : |r(z)|<1\}$ (resp. $\{z : |r(z)|>1\}$) fixing $1$ and $0$ (resp. $\infty$). 
\end{notation}

\begin{thm}\label{maliksthm} Given Blaschke products $A$, $B$ solving $A\circ k=B$, there exists a rational map $r$ with a Jordan curve lemniscate satisfying $r\circ\phi_{\textrm{in}}=A$, $r\circ\phi_{\textrm{out}}=B$. Moreover, the lemniscate $r^{-1}(\mathbb{T})$ depends only on the homeomorphism $k$ and not otherwise on $A$, $B$.
\end{thm}

\begin{proof} We consider the surface 
\begin{equation}\label{maliksrs} S:= \overline{\mathbb{D}}\sqcup\overline{\mathbb{D}^*}\mathbin{\big/}{\sim_k}  \end{equation}
where $\sim_k$ identifies $z\in\mathbb{T}$ with $k(z)\in\mathbb{T}$ along the two copies of $\mathbb{T}$. One readily defines charts on $S$ with coordinate transition maps which are analytic so that $S$ is a Riemann surface, and since $S$ is homeomorphic to $\Chat$, $S$ must be conformally equivalent to $\Chat$ by the uniformization theorem, and we may choose the uniformizing map to fix $0$, $1$, $\infty$. By definition of $\sim_k$, the piecewise definition $F(z):=A(z)$ for $z\in\overline{\mathbb{D}}$ and $F(z):=B(z)$ for $z\in\overline{\mathbb{D}^*}$ is well-defined on $S$, and pre-composing $F$ with the uniformizing map $\Chat\mapsto S$ gives the desired rational function $r$. The lemniscate $r^{-1}(\mathbb{T})$ depends only on the uniformizing map $S\mapsto\Chat$, which by (\ref{maliksrs}) evidently depends only on $\sim_k$ and not otherwise on $A$, $B$.
\end{proof}

\begin{rem} Another proof of Theorem \ref{maliksthm} may be deduced from existence and uniqueness of solutions to the Beltrami Equation. One considers a quasiconformal extension 
\begin{equation} \widehat{k}: \mathbb{D}^*\rightarrow \mathbb{D}^*
\end{equation}
of $k$ fixing $\infty$. Let $g:\Chat\rightarrow\Chat$ denote the quasiregular map:
\begin{equation}  g:= \begin{cases} B\circ\widehat{k} \textrm{ in } \mathbb{D}^* \\  A \hspace{6.5mm} \textrm{ in } \mathbb{D}, \end{cases}
\end{equation}
and let $\mu_g$ denote the Beltrami coefficient of $g$, that is $\mu_g:=g_{\overline{z}}/g_z$. By the Measurable Riemann Mapping Theorem (asserting existence and uniqueness of solutions to the Beltrami Equation), there exists a unique quasiconformal mapping $\phi:\Chat\rightarrow\Chat$ fixing $0$, $1$, $\infty$ so that $\phi_{\overline{z}}/\phi_{z}=\mu$. Consequently, the map defined by
\begin{equation} r(z)=g\circ\phi^{-1}(z)%:=\begin{cases} C\circ\hat{f} \textrm{ in } \mathbb{D}^* \\  B \hspace{8.5mm} \textrm{ in } \mathbb{D}. \end{cases}
\end{equation}
is holomorphic on $\Chat$ and hence a rational map which, moreover, one may check has the desired properties. Moreover $r^{-1}(\mathbb{T})$ depends only on $k$ and not otherwise on $A$, $B$.
\end{rem}

The relevance of Theorem \ref{maliksthm} to the proof of Theorem \ref{uniquenessconcmalik} is that, upon applying Theorem \ref{maliksthm} to each of the two pairs $A$, $B$ and $A^*$, $B^*$, one has two rational functions $r$ and $r^*$ sharing the same lemniscate $L$. Besides the uniformization theorem, the other ingredient in the proof of Theorem \ref{uniquenessconcmalik} is an application of L\"uroth's theorem due to \cite{MR788850}, which will allow us to deduce that all rational functions having lemniscate $L$ (in particular $r$ and $r^*$) must factor through an (essentially unique) minimal degree rational map. First let us recall the statement of L\"uroth's theorem. 

\begin{thm}\label{lurothclassical} \emph{(L\"uroth's Theorem)} Let $\mathbb{K}=\mathbb{R}$ or $\mathbb{C}$. Every subfield $F\subseteq\emph{Rat}_{\mathbb{K}}$ containing all the constant rational functions $c\in \mathbb{K}$ is generated by a single $R_0\in F$, that is:
\begin{equation} F=\{R\circ R_0 : R\in\emph{Rat}_{\mathbb{K}}\} \end{equation}
\end{thm}

\begin{notation} We denote by \textrm{QFBP} the space of quotients of finite Blaschke products; these are exactly the functions $r\in\textrm{Rat}_{\mathbb{C}}$ satisfying $r(\mathbb{T})\subset\mathbb{T}$. 
\end{notation}

\begin{thm}\label{SS_thm} \cite{MR788850} There exists a rational function $g$ satisfying $g^{-1}(\mathbb{T})\supset L$ so that:
\begin{equation}\label{ssthmcond} \{ s \in \emph{Rat} : s^{-1}(\mathbb{T})\supset L\} =   \{ Q\circ g : Q\in\emph{QFBP}  \}.  \end{equation}
\end{thm}

\begin{proof}  Let $\phi$ be a Riemann map of the upper half-plane onto the bounded component of $\Chat\setminus L$, we have that L\"uroth's Theorem with $\mathbb{K}=\mathbb{R}$ applies to the subfield
\begin{equation}\label{applyinglurothtothis} \mathcal{R}:=\{ R\in\textrm{Rat}_{\mathbb{R}} :  R\circ\phi^{-1}\in\textrm{Rat}_{\mathbb{C}} \}.  \end{equation}
of $\textrm{Rat}_{\mathbb{R}}$ and hence is generated by a single element $f\in\mathcal{R}$. 

Let $M$ be a M\"obius transformation satisfying $M(\mathbb{T})=\widehat{\mathbb{R}}$. If $s \in \textrm{Rat}$ satisfies $s^{-1}(\mathbb{T})\supset L$, then $M\circ s\circ\phi\in \mathcal{R}$ so that 
\[ M\circ s\circ\phi = R\circ f \] 
for some real rational $R$, or equivalently
\[ s=M^{-1}\circ R\circ M\circ M^{-1}\circ f\circ \phi^{-1}. \] 
Note that $M^{-1}\circ \textrm{Rat}_{\mathbb{R}}\circ M$ coincides with the collection $\textrm{QFBP}$, so $M^{-1}\circ R\circ M\in\textrm{QFBP}$. Moreover, $f\circ \phi^{-1}\in\textrm{Rat}$ by (\ref{applyinglurothtothis}), and $L\subset (M^{-1}\circ f\circ \phi^{-1})^{-1}(\mathbb{T})$. Thus
\[ g:=M^{-1}\circ f\circ \phi^{-1} \]
satisfies the conclusion of the Theorem.
\end{proof}

\begin{rem} We remark that the field of rational functions $\mathcal{S}$ which are real on a given lemniscate $L$ is a subfield of $\textrm{Rat}_\mathbb{C}$ but $\mathbb{C}\not\subseteq \mathcal{S}$, so L\"uroth's theorem with $\mathbb{K}=\mathbb{C}$ does not apply to $\mathcal{S}$. Neither is $\mathcal{S}$ contained in $\textrm{Rat}_\mathbb{R}$, so L\"uroth's theorem with $\mathbb{K}=\mathbb{R}$ also does not apply. This explains the need for the argument in the proof of Theorem \ref{SS_thm} transforming $\mathcal{S}$ into a subfield of $\textrm{Rat}_\mathbb{R}$.
\end{rem}

\begin{rem} A different proof of Theorem \ref{SS_thm} may be given by applying a real-variant of L\"uroth's Theorem (see the main result of \cite{MR1485305}) directly to $\mathcal{S}$. We also remark that Theorem \ref{SS_thm} was improved upon considerably in the recent work \cite{MR4930120}, where it is shown that in fact rational functions are determined up to $\textrm{QFBP}$ by only finitely many points on their lemniscates. 
\end{rem}

\noindent Recall that $r$, $r^*$ denote rational functions obtained by welding $A$, $B$ and $A^*$, $B^*$ as in Theorem \ref{maliksthm}, and that $r$, $r^*$ share a lemniscate $L$.

\begin{cor}\label{blaschkefactoredout} There exists a rational function $s$ and Blaschke products $C$, $C^*$ so that $r=C\circ s$ and $r^*=C^*\circ s$. 
\end{cor}

\begin{proof} By Theorem \ref{SS_thm}, there exists $C \in \textrm{QFBP}$ so that $r=C\circ s$. If $1/C$ is a Blaschke product, then $r=C\circ s=1/C\circ 1/s$ (since Blaschke products commute with $z\mapsto 1/z$). Thus the conclusion of the Corollary will follow if we can show that either $C$ or $1/C$ is a Blaschke product. Equivalently, we need to show $C^{-1}(\mathbb{T})=\mathbb{T}$. Suppose by way of contradiction $C^{-1}(\mathbb{T})\not=\mathbb{T}$, or equivalently
\begin{equation}\label{SWOC} \mathbb{T} \subsetneq C^{-1}(\mathbb{T}).\end{equation}
Note that 
\begin{equation}\label{whence} r^{-1}(\mathbb{T})=(C\circ s)^{-1}(\mathbb{T})= s^{-1}(C^{-1}(\mathbb{T})). \end{equation}
  Since $L\subseteq s^{-1}(\mathbb{T})$, it follows from (\ref{SWOC}) that $L \subsetneq s^{-1}(C^{-1}(\mathbb{T}))$, whence from (\ref{whence}) we see that $L\subsetneq r^{-1}(\mathbb{T})$ contradicting that $L=r^{-1}(\mathbb{T})$.
Thus $C$ is a Blaschke product, and the same argument applies to $C^*$. 
\end{proof}

Since $C^{-1}(\mathbb{T})=\mathbb{T}$, we see that $s$ also shares the same lemniscate $L$ with $r$, $r^*$, and so recalling $\phi_{\textrm{in}}$, $\phi_{\textrm{out}}$ from Notation \ref{phi_notation}, we have that 
\begin{gather}\label{a1b1defn} \nonumber A_1:=s\circ\phi_{\textrm{in}}, \\ \nonumber B_1:=s\circ\phi_{\textrm{out}},
\end{gather}
 are Blaschke products. We are now able to finish the proof of Theorem \ref{uniquenessconcmalik}. 

\begin{prop} The three pairs $A$, $B$ and $A^*$, $B^*$ and $A_1$, $B_1$ satisfy
\begin{equation}\label{matrix_versionagain} \begin{bmatrix} A & B \\ A^* & B^* \end{bmatrix} = \begin{bmatrix}  C \circ A_1 \hspace{5mm}  C \circ B_1 \\
 C^* \circ A_1 \hspace{4mm} C^* \circ B_1 \end{bmatrix}.
\end{equation}
\end{prop}

\begin{proof} Since  $A=r\circ\phi_{\textrm{in}}$ by Theorem \ref{maliksthm} and $r=C\circ s$ by Corollary \ref{blaschkefactoredout}, we have 
\[ A=r\circ\phi_{\textrm{in}}=C\circ s\circ\phi_{\textrm{in}} =: C\circ A_1 \] 
and similarly for $B$, $A^*$ and $B^*$.
\end{proof}

\noindent We conclude this Section with two remarks.

\begin{rem} There is an intriguing analogy to be made between the question of to what extent a rational function is determined by its lemniscate (discussed in this Section) and to what extent a rational function is determined by its Julia set. The latter question (and variants thereof) are well-studied: see for instance \cite{MR1376996},  \cite{MR3276598}, \cite{MR3539379}, \cite{MR3816522} and \cite{MR4181717}, although there still remain open questions. L\"uroth's Theorem and its variants do not seem to be applicable in the dynamical setting, although on the other hand the non-dynamical welding of Blaschke products appearing in Theorem \ref{maliksthm} has a well-known dynamical analogue \cite{mcmullen1985simultaneous} (see also Section 7.4 of \cite{MR3445628}). 
\end{rem}

\begin{rem} We also remark that while we have described non-uniqueness for solutions to $A\circ k=B$ in Theorem \ref{uniquenessconcmalik}, the corresponding question of describing non-uniqueness of solutions to the welding equation $\phi^{\gamma}\circ k = \psi^{\gamma}$ of Definition \ref{weldingeqtncurves} is still open. For quasisymmetric $k$, existence and uniqueness is understood (see \cite{MR1966191} and \cite{MR3272752} as well as \cite{MR2892610} for a random analogue), but for less regular $k$ the situation is not well-understood as $\phi^{\gamma}\circ k = \psi^{\gamma}$ can admit solutions for quite irregular $k$ (see \cite{MR2373370}), and for such $k$ uniqueness also fails quite dramatically: see \cite{MR3753187}, \cite{rodriguez2025circlehomeomorphismcompositionweldings} and \cite{rodriguez2026flexiblecurveshausdorffdimension}. 
\end{rem}

\section{Proof of Theorem \ref{mon_thm}.}\label{mainsection}

%In this Section, we will focus largely on the proof of the $\impliedby$ direction in Theorem \ref{mon_thm}; this is the more interesting implication in view of the Introduction. The proof of the $\implies$ direction of Theorem \ref{mon_thm} is straightforward; see Remark \ref{} at the end of this Section for details. 

Recall Remark \ref{mon_group_defn} and Definition \ref{pre_reltn_defn} from the Introduction. We will denote by $\mathcal{P}$ the partition induced by the equivalence relation $\sim$ of Definition \ref{pre_reltn_defn}. We now begin proving that the monodromy groups of $A$ and $B$ both respect $\mathcal{P}$, in the following sense.

\begin{definition}\label{respectdefn} The monodromy group of $A$ is said to \emph{respect} a partition of $\{1, ..., d\}$ if for every monodromy group element $g$ and partition element $P$, there exists a partition element $P'$ so that $gP:=\{gp: p\in P\}\subseteq P'$.
\end{definition} 

\noindent Partitions respected by a monodromy group are called \emph{systems of imprimitivity} (see \cite{MR2557404} for an overview and \cite{MR2774592} for work related to Theorem \ref{mon_thm}).

 \begin{lem}\label{case0lem} Suppose that $j\simeq k$ and $g$ is a loop in $\mathbb{D}$. Then $g_B^*(j)\simeq g_A^*(k)$. 
  \end{lem}
 
 \begin{proof} Since $j\simeq k$ there is a loop $\gamma\subset\mathbb{D}$ and $i$ so that $\gamma_B^*(i)=j$, $\gamma_A^*(i)=k$. Consider the loop $g\gamma$ (in other words follow the loop $\gamma$ and then $g$). Then 
\[ (g\gamma)_B^*(i)=g_B^*(\gamma_B^*(i))=g_B^*(j), \] 
and
\[ (g\gamma)_A^*(i)=g_A^*(\gamma_A^*(i))=g_A^*(k), \] 
so that $g_B^*(j)\simeq g_A^*(k)$ by definition of $\simeq$.
 \end{proof}

% Suppose $b_j$, $b_k$ lie in the same partition. Then there is a loop $g$ and an arc $\delta$ of $C^{-1}(g)$ so that $\delta$ begins at say $c_i$, and terminates at $c_j$ but analytic continuation of $(f_0, c_i)$ along $\delta$ yields a branch $(f_n, c_j)$ with $f_n(c_j)=b_k$ (really there could be a chain of such paths I think.) We need to show $g_B^*(j) \sim g_B^*(k)$ and $g_C^*(j) \sim g_C^*(k)$.
 
 \begin{prop}\label{case1prop} Suppose that $g$ is a loop in $\mathbb{D}$ so that $g_A^*=g_B^*$, and assume $j\simeq k$. Then $g_A^*(j)\simeq g_A^*(k)$ and hence also $g_B^*(j)\simeq g_B^*(k)$. 
 \end{prop}

\begin{proof} By Lemma \ref{case0lem}, we have that $g_B^*(j)\simeq g_A^*(k)$. Since $g_B^*=g_A^*$ by assumption, we conclude that $g_A^*(j)\simeq g_A^*(k)$ and $g_B^*(j)\simeq g_B^*(k)$, as needed.
\end{proof}

%\begin{rem} In the previous paragraph we showed an explicit arc $\delta\subset C^{-1}(g\gamma)$ so that analytic continuation of $f$ along $\delta$ led to $D(g_C^*(b_j))=D(g_C^*(b_k))$.As a different approach to potentially keep in mind I can try and define the partition to be $b_j\sim b_k$ iff $D$ must identify $b_j$, $b_k$ in order for $D\circ f$ to be a Blaschke product. \end{rem}
 
%Now let us consider the case that $g$ induces different monodromy group elements $g_B$ for $B$ as opposed to $g_C$ for $C$. Let us first prove that $g_C$ satisfies $g_C(P)\subseteq \tilde{P}$. Again assume $j\sim k$ so that there is an arc $\delta$ in $C^{-1}(g)$ starting at $c_i$ and terminating at $c_j$ but analytic continuation of $(f_0, c_i)$ along $\delta$ yields a branch $(f_n, c_j)$ with $f_n(c_j)=b_k$ (really there could be a chain of such paths I think.) 

 \begin{prop}\label{case2prop} Suppose that $g$ is a loop in $\mathbb{D}$ so that $g_A^*\not=g_B^*$, and assume $j\simeq k$. Then $g_B^*(j)\sim g_B^*(k)$ and $g_A^*(j)\sim g_A^*(k)$.
 \end{prop}
 
 \begin{proof} By Lemma \ref{case0lem}, we have that $g_B^*(j)\simeq g_A^*(k)$. On the other hand, $g_B^*(k)\sim g_A^*(k)$ since either  $g_B^*(k)=g_A^*(k)$ (in which case $g_B^*(k)\sim g_A^*(k)$ by reflexivity) or else $g_B^*(k)\not=g_A^*(k)$ in which case $g_B^*(k)\simeq g_A^*(k)$ by definition of $\simeq$. Thus we have $g_B^*(j)\simeq g_A^*(k)$ and $g_B^*(k)\sim g_A^*(k)$ so that $g_B^*(j)\sim g_B^*(k)$ (by transitivity).
 
 The proof that $g_A^*(j)\sim g_A^*(k)$ is similar: $g_B^*(j)\sim g_A^*(j)$ since either  $g_B^*(j)=g_A^*(j)$ (in which case $g_B^*(j)\sim g_A^*(j)$ by reflexivity) or else $g_B^*(j)\not=g_A^*(j)$ in which case $g_B^*(j)\simeq g_A^*(j)$ by definition of $\simeq$. Combining this with $g_B^*(j)\simeq g_A^*(k)$ yields $g_A^*(j)\sim g_A^*(k)$.
 \end{proof}
 
\noindent We may now conclude the following. 

\begin{cor} The monodromy groups of $A$ and $B$ both respect $\mathcal{P}$. 
\end{cor}

\begin{proof} Propositions \ref{case1prop} and \ref{case2prop} demonstrate that 
\[ j\simeq k \implies g_A^*(j)\sim g_A^*(k),\]
 and since $\sim$ is defined via chains of $\simeq$-equivalences, we have that 
 \[ j\sim k \implies g_A^*(j)\sim g_A^*(k),\]
  and similarly for the monodromy group of $B$. 
\end{proof}

%A polynomial version of the following Theorem is usually attributed to Ritt \cite{MR1501189}, but the  independently discovered in a context much closer to our own by Cowen \cite{cowen2012finiteblaschkeproductscompositions} (Ritt worked with polynomials whereas Cowen worked with Blaschke products). We sketch a proof both for the sake of the reader, and because it will later be useful to be able to refer to steps in the proof. We follow \cite{MR3753897} (which in turn follows \cite{cowen2012finiteblaschkeproductscompositions}). 

%\begin{rem}\label{assumption_to_be_rem} For the remainder of this Section we will assume that $L$ is a Jordan curve. \emph{I need to come back and remove this assumption later.}
%\end{rem}

\begin{rem} The following Lemmas \ref{rittsthmpart1'}, \ref{rittsthm} are due to \cite{cowen2012finiteblaschkeproductscompositions}. The work \cite{cowen2012finiteblaschkeproductscompositions} is unpublished for reasons explained therein; an excellent account of \cite{cowen2012finiteblaschkeproductscompositions} may be found in \cite{MR3753897}, and related ideas are also found in the much older \cite{MR1501189}. We nevertheless include proofs of Lemmas \ref{rittsthmpart1'}, \ref{rittsthm} (following \cite{MR3753897} which in turn follows \cite{cowen2012finiteblaschkeproductscompositions}) to give a self-contained account of the proof of Theorem \ref{mon_thm}, but also because we will later refer to some steps in the proofs. % for the reader's sake, as the ideas are central to the proof of Theorem \ref{mon_thm}.
\end{rem}

% influenced by the work of \cite{cowen2012finiteblaschkeproductscompositions} (the work \cite{cowen2012finiteblaschkeproductscompositions} is unpublished for reasons explained therein; a detailed account of \cite{cowen2012finiteblaschkeproductscompositions} may also be found in \cite{MR3753897}) - related ideas are also found in  \cite{MR1501189}. When $L=\mathbb{T}$ and $B$, $C$ are Blaschke products the proofs of Lemma \ref{rittsthmpart1'} and Theorem \ref{rittsthm} are the same as in \cite{cowen2012finiteblaschkeproductscompositions}, but new ideas are needed in what follows for dealing with non-circular or disconnected $L$, as well as for dealing with non-Blaschke $B$, $C$. 

\begin{notation} Denote by $g_j$ the branch of $B^{-1}$ taking $\zeta$ to $b_j$ ($g_j$ is defined in a neighborhood of $\zeta$). We denote by $P\in\mathcal{P}$ the partition element containing $1$. 
\end{notation}

\begin{lem}\label{rittsthmpart1'} The mapping
\begin{equation}\label{defnofA} B_1(z):= \prod_{j\in P} g_j\circ B(z), \end{equation}
defined for all $z$ sufficiently close to $b_1$, extends to a proper map $B_1: \mathbb{D} \rightarrow\mathbb{D}$ of degree $|P|$. 
\end{lem}

\begin{proof} First note that
\begin{equation}\label{d1formula} B_1(b_1) = \prod_{j\in P} b_j.\end{equation}
 Moreover, the assumption that the monodromy group respects the partition $\mathcal{P}$ implies that if $j\in P$ and $\gamma\subset \mathbb{D}$ is a loop based at $b_1$ avoiding $B^{-1}(\textrm{CV}(B))$, then analytic continuation of $(g_j\circ B, b_1)$ along $\gamma$ leads to a branch $(g_{k}\circ B, b_{k})$ where $j\sim k$, in other words $k\in P$. Thus analytic continuation of $B_1$ along any loop in $\mathbb{D}$ based at $b_1$ leads to the same terminating value (\ref{d1formula}): the terms in the product are rearranged but the value of the product over all terms remains the same. Thus, analytic continuation gives a single-valued extension of $B_1$ to a map $B_1: \mathbb{D}\setminus B^{-1}(\textrm{CV}(B)) \rightarrow \mathbb{D}$. In fact, since $B: \mathbb{D} \rightarrow\mathbb{D}$ is proper, we see that $B_1$ is unimodular exactly on $\mathbb{T}$. Thus the singularities at $B^{-1}(\textrm{CV}(B))$ are removable and $B_1$ extends to a proper map $B_1: \mathbb{D}\rightarrow\mathbb{D}$. It may be verified that $B_1$ has increasing argument along $B_1^{-1}(|\zeta|\mathbb{T})$, and hence $B_1$ takes on the value $\zeta$ exactly at $\{b_i\}_{i\in P}$; in particular $\deg(B_1)=|P|$.
\end{proof}

\noindent Since the only proper holomorphic maps $\mathbb{D}\mapsto\mathbb{D}$ are Blaschke products, we note that $B_1$ is a Blaschke product.

\begin{lem}\label{rittsthm} There is a Blaschke product $B_2$ so that $B=B_2\circ B_1$.%; moreover $\{c_j\}_{j\in P}$ is a fiber of $B_2$ for each $P\in\mathcal{P}$. 
\end{lem}

\begin{proof} As before, we denote by $P\in\mathcal{P}$ the partition element containing $b_1$. Let $P'\in\mathcal{P}$ be a partition element, and $k\in P'$. Denote by $\eta\subset \mathbb{D}$ an arc connecting $b_1$ to $b_k$, and let $j\in P$. Then, since $B$ respects $\mathcal{P}$ and $j\in P$, analytic continuation of $(g_j\circ B, b_1)$ along $\eta$ leads to a branch $(g_{\ell}\circ B, b_{\ell})$ where $\ell\sim k$, in other words $\ell\in P'$. Thus, recalling (\ref{defnofA}), we see that analytically continuing $B_1$ along $\eta$ leads to the identity
\begin{equation}\label{thisisafiber} B_1(b_k)=\prod_{\ell\in P'} b_\ell.\end{equation}
In particular, as the right-hand side of (\ref{thisisafiber}) depends on $P'$ but not on $k\in P'$, the points $\{b_k\}_{k\in P'}$ lie in a common fiber of $B_1$ over the value (\ref{thisisafiber}): moreover there are no other fiber elements since $|P'|=|P|=\deg(B)$ (the fact that every partition element has the same size is a simple algebraic consequence of Definition \ref{respectdefn} deduced from transitivity of the monodromy group - see for instance Section 9.5 of \cite{MR3753897} for further discussion).

Let $B_2:\mathbb{D}\rightarrow\mathbb{D}$ be a Blaschke product of degree $\deg(B)/|P|$ so that
\begin{equation}\label{thisiswhereidefineb2} B_2\left( \prod_{j\in P'} b_j \right) = \zeta \textrm{ for each } P'\in \mathcal{P}.
\end{equation}
Then $B_2\circ B_1$ is of degree $\deg(A)$ and satisfies 
\[(B_2\circ B_1)^{-1}(\zeta)= \{b_j\}_{j=1}^d. \]
Since $B$, $B_2\circ B_1$ are proper onto $\mathbb{D}$ of the same degree, and share a common fiber over $\zeta\in\mathbb{D}$, we conclude that $M\circ B_2\circ B_1=B$ for some M\"obius transformation $M$ preserving $\mathbb{D}$, and $M$ may be absorbed into $B_2$.
\end{proof}

\noindent Combining Lemmas \ref{rittsthmpart1'} and \ref{rittsthm} we have the following.

\begin{thm}\label{ccompositionthm} There exist Blaschke products $B_1$ (resp. $B_2$) of degree $|P|$ (resp. $d/|P|$) so that $B=B_2\circ B_1$.
\end{thm}

\noindent Recalling that the partition $\mathcal{P}$ is the same for both $A$ and $B$, we see that exactly the same argument replacing $B$ with $A$ yields the following. 

\begin{thm}\label{bcompositionthm} There exists Blaschke products $A_1$ (resp. $A_2$) of degree $|P|$ (resp. $d/|P|$) so that $A=A_2\circ A_1$.
\end{thm}

We will now prove that the second factors $A_2$, $B_2$ in the decompositions $A=A_2\circ A_1$, $B=B_2\circ B_1$ coincide up to $\textrm{M\"ob}$. Recall from the proofs of Lemmas \ref{rittsthmpart1'} and \ref{rittsthm} that for each $P'\in\mathcal{P}$, the points $\{b_i\}_{i\in P'}$ (resp. $\{a_i\}_{i\in P'}$) constitute a common fiber of $B_1$ (resp. $A_1$) over the value $b(P'):=\prod_{i\in P'}b_i$ (resp $a(P'):=\prod_{i\in P'}a_i$). Recall also that $B_2$ (resp. $A_2$) maps each $b(P')$ (resp. $a(P')$) to $\zeta$, and that the element of $\mathcal{P}$ containing $1$ is denoted $P$.

% Recall our notation $P\in\mathcal{P}$ for the partition element containing $1$. From the proofs of Lemmas \ref{rittsthmpart1'} and \ref{rittsthm}, we see that the points $\{b_i\}_{i\in P}$ (resp. $\{a_i\}_{i\in P}$) lie in a common fiber of $B_1$ (resp. $A_1$) over the value $z:=\prod_{i\in P}b_i$ (resp $w:=\prod_{i\in P}a_i$).

\begin{lem}\label{analytic_cont_along_gamma} Let $\gamma$ be an oriented loop based at $b(P)$ so that $B_2(\gamma)$ avoids the critical values of $A_2$, and let $g$ be the branch of $A_2^{-1}$ mapping $\zeta$ to $a(P)$. Then analytic continuation of $g\circ B_2$ along the loop $\gamma$ terminates at $a(P)$. 
\end{lem}

\begin{proof} Let $\Gamma:=B_2(\gamma)$. Consider the monodromy group element $\Gamma_B$. Let $i\in P$. Since $\gamma$ terminates at $b(P)$ we have $\Gamma_B(i)\in P$. 

On the other hand, since $g$ is a local inverse to $A_2$, the terminating value of the analytic continuation of $g\circ B_2$ along $\gamma$ must be $a(P')$ for some $P'\in\mathcal{P}$ (since $\{a(P')\}_{P'\in\mathcal{P}}$ is exactly the fiber of $A_2$ over the value $\zeta$). This means that $\Gamma_A(i)\in P'$. 

But then by definition of $\sim$ we have $\Gamma_B(i)\sim\Gamma_A(i)$ and so $P=P'$ must be the same partition element. In other words, the terminating value of the analytic continuation of $g\circ B_2$ along $\gamma$ is $a(P)$, as needed. 
\end{proof}

\begin{cor}\label{extendstomobiuscor} The Blaschke products $A_2$, $B_2$ satisfy $A_2 \circ M = B_2$ for some $M\in\emph{M\"ob}_{\mathbb{D}}$. 
\end{cor}

\begin{proof} By Lemma \ref{analytic_cont_along_gamma}, analytic continuation of $g\circ B_2$ along any loop based at $b(P)$ leads to the same terminating value $a(P)$. Thus analytic continuation of $g\circ B_2$ gives a well-defined extension of $g\circ B_2$ to a proper map $\mathbb{D}\mapsto\mathbb{D}$, and this proper map is $1:1$ on $\mathbb{T}$ hence a M\"obius transformation $M$. Thus $g\circ B_2=M$ and so $B_2=A_2\circ M$ as needed.
\end{proof}

\noindent The $\impliedby$ direction of Theorem \ref{mon_thm} is now proven with $C:=B_2$ (note that $\deg(C)=d/|P|$ by Theorem \ref{ccompositionthm}, and $d/|P|>1$ since $\sim$ is assumed to be non-trivial), with $B_1$ as already defined, and by absorbing $M$, $M^{-1}$ (where $M$ is as in Corollary \ref{extendstomobiuscor})  into $A_2$, $A_1$:
\[ A=A_2\circ A_1=A_2\circ M\circ M^{-1}\circ A_1=C\circ M^{-1}\circ A_1. \] 

\begin{rem} The $\implies$ direction of Theorem \ref{mon_thm} is simpler to prove: if $A=C\circ A_1$, $B=C\circ B_1$ for $\deg(C)>1$, then one chooses any ordering $a_1$, ..., $a_d$ and $b_1$, ..., $b_d$ so that $a_i$ and $b_i$ both lie in a fiber over a common point in $C^{-1}(\zeta)$ for each $1\leq i \leq d$. Then the equivalence relation $\sim_1$ of $\{1, ..., d\}$ defined by $i\sim_1j$ if and only if $A(a_i)=A(a_j)$ (which also occurs if and only if $B(b_i)=B(b_j)$) defines a non-trivial equivalence relation satisfying that $\gamma_A^*(i)\sim_1\gamma_B^*(i)$ for each $i$. Since $\sim_1$ is a non-trivial equivalence relation containing the relation $\simeq$ of Definition \ref{pre_reltn_defn}, the equivalence relation $\sim$ of Definition \ref{pre_reltn_defn} generated by $\simeq$ is also non-trivial.
\end{rem}

%\begin{rem} Although our interest in Theorem \ref{mon_thm} is primarily with regards to Blaschke products, it may likely be formulated in other settings such as 
%\end{rem}


\begin{thebibliography}{AJKS11}

\bibitem[AJKS11]{MR2892610}
Kari Astala, Peter Jones, Antti Kupiainen, and Eero Saksman.
\newblock Random conformal weldings.
\newblock {\em Acta Math.}, 207(2):203--254, 2011.

\bibitem[BEL25]{MR4880202}
Christopher~J. Bishop, Alexandre Eremenko, and Kirill Lazebnik.
\newblock On the shapes of rational lemniscates.
\newblock {\em Geom. Funct. Anal.}, 35(2):359--407, 2025.

\bibitem[BF14]{MR3445628}
Bodil Branner and N\'{u}ria Fagella.
\newblock {\em Quasiconformal surgery in holomorphic dynamics}, volume 141 of
  {\em Cambridge Studies in Advanced Mathematics}.
\newblock Cambridge University Press, Cambridge, 2014.
\newblock With contributions by Xavier Buff, Shaun Bullett, Adam L. Epstein,
  Peter Ha\"{\i}ssinsky, Christian Henriksen, Carsten L. Petersen, Kevin M.
  Pilgrim, Tan Lei and Michael Yampolsky.

\bibitem[Bis07]{MR2373370}
Christopher~J. Bishop.
\newblock Conformal welding and {K}oebe's theorem.
\newblock {\em Ann. of Math. (2)}, 166(3):613--656, 2007.

\bibitem[BLM16]{MR3539379}
Mario Bonk, Mikhail Lyubich, and Sergei Merenkov.
\newblock Quasisymmetries of {S}ierpi\'nski carpet {J}ulia sets.
\newblock {\em Adv. Math.}, 301:383--422, 2016.

\bibitem[Cow12]{cowen2012finiteblaschkeproductscompositions}
Carl~C. Cowen.
\newblock Finite blaschke products as compositions of other finite blaschke
  products.
\newblock \url{https://arxiv.org/abs/1207.4010}, 2012.

\bibitem[EKS11]{MR2868587}
P.~Ebenfelt, D.~Khavinson, and H.~S. Shapiro.
\newblock Two-dimensional shapes and lemniscates.
\newblock In {\em Complex analysis and dynamical systems {IV}. {P}art 1},
  volume 553 of {\em Contemp. Math.}, pages 45--59. Amer. Math. Soc.,
  Providence, RI, 2011.

\bibitem[FKV18]{MR3784168}
Anastasia Frolova, Dmitry Khavinson, and Alexander Vasil'ev.
\newblock Polynomial lemniscates and their fingerprints: from geometry to
  topology.
\newblock In {\em Complex analysis and dynamical systems}, Trends Math., pages
  103--128. Birkh\"{a}user/Springer, Cham, 2018.

\bibitem[GMR17]{MR3753897}
Stephan~Ramon Garcia, Javad Mashreghi, and William~T. Ross.
\newblock Finite {B}laschke products: a survey.
\newblock In {\em Harmonic analysis, function theory, operator theory, and
  their applications}, volume~19 of {\em Theta Ser. Adv. Math.}, pages
  133--158. Theta, Bucharest, 2017.

\bibitem[Ham02]{MR1966191}
D.~H. Hamilton.
\newblock Conformal welding.
\newblock In {\em Handbook of complex analysis: geometric function theory,
  {V}ol.\ 1}, pages 137--146. North-Holland, Amsterdam, 2002.

\bibitem[Kir87]{MR902292}
A.~A. Kirillov.
\newblock K\"ahler structure on the {$K$}-orbits of a group of diffeomorphisms
  of the circle.
\newblock {\em Funktsional. Anal. i Prilozhen.}, 21(2):42--45, 1987.

\bibitem[LM18]{MR3816522}
Mikhail Lyubich and Sergei Merenkov.
\newblock Quasisymmetries of the basilica and the {T}hompson group.
\newblock {\em Geom. Funct. Anal.}, 28(3):727--754, 2018.

\bibitem[LP97]{MR1376996}
G.~Levin and F.~Przytycki.
\newblock When do two rational functions have the same {J}ulia set?
\newblock {\em Proc. Amer. Math. Soc.}, 125(7):2179--2190, 1997.

\bibitem[Mar11]{MR2858965}
Donald~E. Marshall.
\newblock Conformal welding for finitely connected regions.
\newblock {\em Comput. Methods Funct. Theory}, 11(2):655--669, 2011.

\bibitem[McM85]{mcmullen1985simultaneous}
Curtis~T. McMullen.
\newblock Simultaneous uniformization of blaschke products.
\newblock Preliminary notes, 1985.

\bibitem[OP25]{MR4930120}
Stepan Orevkov and Fedor Pakovich.
\newblock On intersection of lemniscates of rational functions.
\newblock {\em Arnold Math. J.}, 11(1):7--26, 2025.

\bibitem[Pak09]{MR2557404}
F.~Pakovich.
\newblock Prime and composite {L}aurent polynomials.
\newblock {\em Bull. Sci. Math.}, 133(7):693--732, 2009.

\bibitem[Pak11]{MR2774592}
F.~Pakovich.
\newblock Algebraic curves {$P(x)-Q(y)=0$} and functional equations.
\newblock {\em Complex Var. Elliptic Equ.}, 56(1-4):199--213, 2011.

\bibitem[Pak20]{MR4181717}
F.~Pakovich.
\newblock On rational functions sharing the measure of maximal entropy.
\newblock {\em Arnold Math. J.}, 6(3-4):387--396, 2020.

\bibitem[Rit22]{MR1501189}
J.~F. Ritt.
\newblock Prime and composite polynomials.
\newblock {\em Trans. Amer. Math. Soc.}, 23(1):51--66, 1922.

\bibitem[Rod25]{rodriguez2025circlehomeomorphismcompositionweldings}
Alex Rodriguez.
\newblock Every circle homeomorphism is the composition of two weldings, 2025.

\bibitem[Rod26]{rodriguez2026flexiblecurveshausdorffdimension}
Alex Rodriguez.
\newblock Flexible curves and hausdorff dimension, 2026.

\bibitem[RS97]{MR1485305}
T.~Recio and J.~R. Sendra.
\newblock A really elementary proof of real {L}\"{u}roth's theorem.
\newblock volume~10, pages 283--290. 1997.
\newblock Real algebraic and analytic geometry (Segovia, 1995).

\bibitem[SM06]{Sharon-Mumford}
E.~Sharon and D.~Mumford.
\newblock {2D-Shape} analysis using conformal mapping.
\newblock {\em Int. J. Comput. Vis.}, 70:55--75, 2006.

\bibitem[Sol20]{solynin2020fingerprintslemniscatesquadraticdifferentials}
Alexander~Yu. Solynin.
\newblock Fingerprints, lemniscates and quadratic differentials.
\newblock \url{https://arxiv.org/abs/2011.03855}, 2020.

\bibitem[SS85]{MR788850}
Kenneth Stephenson and Carl Sundberg.
\newblock Level curves of inner functions.
\newblock {\em Proc. London Math. Soc. (3)}, 51(1):77--94, 1985.

\bibitem[SS15]{MR3272752}
Eric Schippers and Wolfgang Staubach.
\newblock A symplectic functional analytic proof of the conformal welding
  theorem.
\newblock {\em Proc. Amer. Math. Soc.}, 143(1):265--278, 2015.

\bibitem[Ye15]{MR3276598}
Hexi Ye.
\newblock Rational functions with identical measure of maximal entropy.
\newblock {\em Adv. Math.}, 268:373--395, 2015.

\bibitem[You16]{MR3447662}
Malik Younsi.
\newblock Shapes, fingerprints and rational lemniscates.
\newblock {\em Proc. Amer. Math. Soc.}, 144(3):1087--1093, 2016.

\bibitem[You18]{MR3753187}
Malik Younsi.
\newblock Removability and non-injectivity of conformal welding.
\newblock {\em Ann. Acad. Sci. Fenn. Math.}, 43(1):463--473, 2018.

\end{thebibliography}
\end{document}